\documentclass[reqno]{amsart}
\usepackage{graphicx}
\usepackage{geometry}
\usepackage[leqno]{amsmath}
\usepackage{amscd}
\usepackage{amsthm}
\usepackage{amsfonts}
\usepackage{amssymb}
\usepackage[pdfencoding=auto,colorlinks=true]{hyperref}
\usepackage{xcolor}
\usepackage{mathrsfs}
\newtheorem{theorem}{Theorem}[section]

\newtheorem{corollary}[theorem]{Corollary}
\newtheorem{proposition}[theorem]{Proposition}
\newtheorem{lemma}[theorem]{Lemma}

\theoremstyle{definition}

\newtheorem{example}[theorem]{Example}

\newtheorem{remark}[theorem]{Remark}
\makeatletter
\newcommand{\leqnomode}{\tagsleft@true}
\newcommand{\reqnomode}{\tagsleft@false}
\makeatother

\DeclareMathOperator{\End}{End}
\DeclareMathOperator{\Aut}{Aut}

\DeclareMathOperator{\GL}{GL}

\DeclareMathOperator{\Perm}{Perm}

\newcommand{\gen}[1]{\langle #1 \rangle} 

\newcommand{\Z}{\mathbb{Z}}

\newcommand{\ak}[1]{{}}
\newcommand{\pt}[1]{{}}

\newcommand{\id}{\mathrm{id}}
\newcommand{\Ab}{\mathrm{Ab}}
\newcommand{\nrm}{\trianglelefteq}
\newtheorem{thmbig}{Theorem}

\newcommand{\fix}{\operatorname{fix} \psi}
\newcommand{\zprod}{\prod_{i\in \Z_n} }

\numberwithin{equation}{section}

\begin{document}
\title{Constructing skew bracoids via abelian maps, and solutions to the {Y}ang-{B}axter equation}
\author{Alan Koch and Paul J.~Truman}
\date{\today}

\begin{abstract} We show how one can use the skew braces constructed using abelian maps to generate families of skew bracoids as defined by Martin-Lyons and Truman. Under certain circumstances, these bracoids give right non-degenerate solutions to the Yang-Baxter equation. 
\end{abstract}
\maketitle

\section{Introduction and Statement of Main Results}

A \textit{skew left brace} is a triple $(G,\cdot,\circ)$ such that $(G,\cdot)$ and $(G,\circ)$ are groups and 
\[g\circ(h\cdot k) = (g\circ h)\cdot g^{-1} \cdot (g\circ k)\]
where $g^{-1}$ is the inverse in $(G,\cdot)$. As is well-known (see, e.g., \cite{SmoktunowiczVendramin18}), skew left braces connect with several areas of mathematics, including Hopf-Galois theory and solutions to the Yang-Baxter equation. In the skew left brace $(G,\cdot,\circ)$ we will refer to $(G,\cdot)$ as the \textit{additive group} and $(G,\circ)$ as the \textit{multiplicative group}. Note that neither group is assumed to be abelian. 

In 2023 the second author, together with Martin-Lyons developed \textit{skew bracoids} (hereafter, \textit{bracoids}), a generalization of skew left braces (hereafter, \textit{braces}). A bracoid is a quintuple $(G,\cdot,N,\star,\odot)$ where $(G,\cdot)$ and $(N,\star)$ are groups, and $G$ acts transitively on $N$ via $\odot$ such that the following \textit{bracoid relation} holds:
\[g \odot (\eta \star \mu) = (g\odot \eta)\star (g\odot e_N)^{-1} \star (g\odot \mu),\ g\in G,\ \eta,\mu \in N.\] Note that one recovers the usual brace relation in the case $G=N$ as sets. Bracoids have applications to Hopf-Galois theory \cite{Martin-LyonsTruman23}, and in some instances can give solutions to the Yang-Baxter equation \cite{ColazzoKochMartin-LyonsTruman24}. Note that the definition of ``bracoid'' here is unrelated to the concept since developed in \cite{ShengTangZhu2024}.

 Braces contain a number of substructures; here, we identify the two most important for the results to follow. Associated to a brace $(G,\cdot,\circ)$ is a homomorphism $\gamma:(G,\circ)\to \Aut(G,\cdot)$ given by $\gamma(g)[h] = g^{-1}(g\circ h)$. If $H\subseteq G$ satisfies $H\le (G,\circ),\ H\nrm(G,\cdot)$, and $\gamma(g)[H]\le H$ for  all $g\in G$ then $H$ is said to be a \textit{strong left ideal} of $(G,\cdot,\circ)$. A strong left ideal $H$ such that $H\nrm (G,\circ)$ is said to be an \textit{ideal} of $(G,\cdot,\circ)$. One can check that if $H$ is a strong left ideal of a brace $(G,\cdot,\circ)$ then the two operations are well-defined on the quotient and we get $(G/H,\cdot,\circ)$ is also a brace.
 
A key technique for constructing bracoids can be found in \cite[Prop. 2.4]{Martin-LyonsTruman23}, where the authors start with a brace $(B,\star,\cdot)$ and take a strong left ideal $A$; doing so produces the bracoid $(B,\cdot,B/A,\star,\odot)$, where $b\odot cA = bcA$.
Thus, a large class of bracoids can be facilitated by identifying the strong left ideals of a known brace. 

Now let $G,G'$ be groups, and let $\psi: G\to G'$ be a homomorphism whose image $\psi(G)$ is abelian. In the case $G'=G$ (as in \cite{Koch21}), such maps are called \textit{abelian maps}, and the set of all such abelian maps is denoted $\Ab(G)$. In the more general setting we consider here, we will continue to use the term ``abelian map'' and we will denote the set of all such maps by $\Ab(G,G')$.

The main result of \cite{Koch21} is to illustrate that $\psi\in \Ab(G)$ gives rise to a binary operation $\circ$ on $G$ defined by
\[g\circ h = g\psi(g^{-1})h\psi(g),\ g,h\in G;\] 
with this definition $(G,\circ)$ is a group and both $(G,\cdot,\circ)$ and $(G,\circ,\cdot)$ form braces. The triple $(G,\cdot,\circ)$ (equivalently, $(G,\circ,\cdot)$) is what Childs calls a \textit{bi-skew brace} in \cite{Childs19}. Associated to $\psi$ is an additional homomorphism (typically not abelian) $\phi: (G,\circ)\to (G,\cdot)$ given by $\phi(g)=g\psi(g^{-1})$ \cite[Prop. 5.5]{Koch22}. This map $\phi$ is crucial in this work, and of course implicitly depends on $\psi$. Note, for example, that $g\circ h = \phi(g)h\psi(g)$.

For $(G,\cdot,\circ)$ any brace, one can form the \textit{opposite brace} $(G,\cdot',\circ)$ where $g\cdot' h = hg$ as shown in \cite{KochTruman20}. Thus, any $\psi\in\Ab(G)$ can give two additional braces, namely $(G,\cdot',\circ)$ and $(G,\circ',\cdot)$. These opposite braces are not typically bi-skew, hence $(G,\cdot',\circ')$ is not in general a brace.

In this work, we seek to connect the theory of abelian maps / bi-skew braces to the construction of bracoids by identifying strong left ideals in the constructed brace. Of course, one difficulty that arises is that the term ``strong left ideal in a bi-skew brace'' is not well-defined, as it depends on which of $(G,\cdot)$ and $(G,\circ)$ is being viewed as the additive group. For clarity, we will differentiate between the two types of strong left ideals when we refer to the brace: a strong left ideal $H$ of $(G,\circ,\cdot)$ will have $H\nrm(G,\circ)$ whereas a strong left ideal of $(G,\cdot,\circ)$ will have $H\nrm (G,\cdot)$. As we will see, it is common for bi-skew braces to have strong left ideals of either type.

In our first main result, we identify precisely the strong left ideals of our bi-skew braces.

\begin{thmbig}\label{biggie}
	Let $(G,\cdot)$ be a group, and $\psi\in\Ab(G)$. Let $H\le G$, and let $C_1$ and $C_2$ be the following two conditions:
	\begin{description}
		\item[$C_1$] $[G,\phi(H)]\le H$;
		\item[$C_2$] $H \nrm G$.
	\end{description}
	Then
	\begin{enumerate}
		\item $H$ is a strong left ideal of the braces $(G,\circ,\cdot)$ and $(G,\circ',\cdot)$ if and only if $C_1$ holds.
		\item $H$ is a strong left ideal of the braces $(G,\cdot,\circ)$ and $(G,\cdot',\circ)$ if and only if $C_2$ holds.
		\item $H$ is an ideal of $(G,\cdot,\circ),\ (G,\cdot,\circ'),$ and $(G,\circ',\cdot)$ if and only if both $C_1$ and $C_2$ hold. 
	\end{enumerate}
\end{thmbig}

The proof of Theorem \ref{biggie} is quickly found by combining the results of Propositions \ref{bigCirc} and \ref{bigDot}. Thanks to the aforementioned \cite[Prop. 2.4]{Martin-LyonsTruman23}, identifying strong left ideals allows us to construct bracoids. Corollaries \ref{bigCirccor} and \ref{bigDotcor} quickly give us the following.

\begin{thmbig}
	Let $(G,\cdot)$ be a group, and $\psi\in\Ab(G)$. Let $H\le G$, and let $C_1$ and $C_2$ be as above.	Then
	\begin{enumerate}
		\item If $C_1$ holds, then $(G,\cdot,G/H,\circ,\odot)$ and $(G,\cdot,G/H,\circ',\odot)$ are bracoids.
		\item If $C_2$ holds, then $(G,\circ,G/H,\cdot,\odot)$ and $(G,\circ,G/H,\cdot',\odot)$ are bracoids.
		\item If $C_1$ and $C_2$ both hold, then $(G/H,\cdot,\circ),\ (G/H,\cdot,\circ')$, and $(G/H,\cdot',\circ)$ are braces. 
	\end{enumerate}
\end{thmbig}

Finally, we turn our attention to constructing set-theoretic solutions to the Yang-Baxter equation. Recall that a \textit{set-theoretic solution to the Yang-Baxter equation} consists of a set $B$ and a map $R:B\times B\to B\times B$ such that
\[(R\times\id)(\id\times R)(R\times \id) = (\id \times R)(r\times \id)(\id\times R): B^3\to B^3.\]
Write a given  solution as $R(x,y)=(\lambda_x(y),\rho_y(x))$. If each $\lambda_x$ is a bijection we say $R$ is \textit{left non-degenerate}; similarly if each $\rho_y$ is a bijection then $R$ is \textit{right non-degenerate}. A solution that is both left non-degenerate and right non-degenerate will be called \textit{non-degenerate}.

It is well-known that a brace will give a bijective, non-degenerate solution to the Yang-Baxter equation; indeed the inverse to the solution arises by considering the opposite brace \cite[Th. 4.1]{KochTruman20}. Unfortunately, 
it is not known whether every bracoid will give a solution to the YBE. However, here we find a special case in each of the two types of strong left ideals above that allow us to construct solutions which are right non-degenerate using a technique developed in \cite{ColazzoKochMartin-LyonsTruman24}. 

\begin{thmbig}
	Let $(G,\cdot)$ be a group.
	\begin{enumerate}
		\item\label{first} If $\psi\in\Ab(G)$ is idempotent, then 
		\[R(x,y) = (\psi(x)\phi(y)\psi(x^{-1}),\psi(x)\phi(y)^{-1}\psi(x^{-1})xy)\]
		is a right non-degenerate solution to the Yang-Baxter equation.
		\item\label{second} If $G= G_1\times G_2$ and $\alpha\in\Ab(G_1,G_2),\ \beta\in\Ab(G_2,G_1)$ then 
		\[R((x_1,x_2),(y_1,y_2)) = \left(\lambda_{(x_1,x_2)}((y_1,y_2)),\rho_{(y_1,y_2)}((x_1,x_2))\right)\]
		is a right non-degenerate solution to the Yang-Baxter equation, where 
				\begin{align*}\lambda_{(x_1,x_2)}(y_1,y_2) &= (e,\alpha(x_1^{-1})y_2\alpha(x_1))\\
			\rho_{(y_1,y_2)}(x_1,x_2)&=\left(\beta(y_2)x_1\beta(x_2^{-1})y_1\beta(x_2y_2^{-1}),\alpha(x_1)^{-1}y_2^{-1}\alpha(x_1)x_2\alpha(x_1)^{-1}y_2\alpha(x_1)\right).
		\end{align*}
	\end{enumerate}
\end{thmbig}

These are proven to be solutions in Propositions \ref{YBE1} and \ref{YBE2}. Neither will be left non-degenerate except in very extreme circumstances. Solutions of type (\ref{first}) will be left non-degenerate if and only if $\psi(g)\ne g$ for all $g\ne e$--that is, $\psi$ is fixed-point free in the sense of \cite{Childs13}, Solutions of type (\ref{second}) will be left non-degenerate if and only if $G_1$ is trivial. 

Throughout, given a bi-skew brace $(G,\cdot,\circ)$ (equivalently,  $(G,\circ,\cdot)$), for $g\in G$ we will denote its inverse in $(G,\cdot)$ by $g^{-1}$ and its inverse in $(G,\circ)$ by $\overline g$. We will denote the identity (which is common to both operations) by $e$, and we will typically write $gh$ for $g\cdot h$. We write $[g,h]=ghg^{-1}h^{-1}$ for the commutator of $g$ and $h$ in $(G,\cdot)$.

 While the theory of abelian maps will work for any group, if $(G,\cdot)$ is itself is abelian then the construction will always yield the trivial brace $(G,\cdot,\cdot)$, so we will implicitly assume $(G,\cdot)$ is nonabelian throughout.

\section{Strong left ideals of $(G,\circ,\cdot)$}

In this section, we identify the strong left ideals of the braces $(G,\circ,\cdot)$ which arise from choosing an abelian map $\psi\in\Ab(G)$. Recall that $\phi:(G,\circ) \to (G,\cdot)$ is a homomorphism defined by $\phi(g) = g\psi(g^{-1})$, and with this notation we may write $g\circ h = g\psi(g^{-1})h\psi(g)=\phi(g)h\psi(g^{-1})$.  

\begin{proposition}\label{bigCirc}
	Let $\psi\in\Ab(G)$, and suppose that $H\le (G,\cdot)$. Then the following are equivalent:
	\begin{enumerate}
		\item \label{un}$H$ is a strong left ideal of $(G,\circ,\cdot)$;
		\item \label{deux}$H$ is a strong left ideal of $(G,\circ',\cdot)$;
		\item \label{trois}$[G,\phi(H)]\le H$.
	\end{enumerate} 
\end{proposition}

\begin{proof}
	
	First, we establish the equivalence of (\ref{un}) and (\ref{deux}). Suppose $H\le(G,\cdot)$ is a strong left ideal of $(G,\circ,\cdot)$. Then $H\nrm(G,\circ)$ and the  map $\gamma_{\circ}:(G,\cdot)\to\Aut(G,\circ)$ given by $\gamma_{\circ}(g)[h] = \overline{g}\circ(gh)$ satisfies $\gamma_{\circ}(g)[H]\subseteq H$ for all $g\in G$. (Note that $\gamma_{\circ}$ looks different from the $\gamma$ given in the definition since we are working with $\circ$ as our additive group here.) As $g\circ' h \circ' \overline g = \overline g \circ h \circ g$ the fact that $H\nrm(G,\circ)$ quickly gives that $H$ is normal in $(G,\circ')$. Now let $g\in G,h\in H$. Then $\overline g\circ (gh) = h'$ for some $h\in H$. Then
	\[\overline g \circ' (gh) = gh \circ \overline g = g\circ \overline g \circ gh \circ \overline g=g\circ h' \circ \overline g \in H,\]
	thus $\overline g \circ' (gh)\in H$ as well and $H$ is a strong left ideal of $(G,\circ',\cdot)$. Interchanging $\circ$ and $\circ'$ shows that (\ref{un}) and (\ref{deux}) are equivalent.

	It remains to show that (\ref{un}) and (\ref{trois}) are equivalent. 
	Let $g\in G,\ h \in H$. It is easy to verify that $\overline g = \psi(g)g^{-1}\psi(g^{-1}) = \phi(g)^{-1}\psi(g^{-1})$, and hence we get
	\begin{align*}
		g\circ h \circ \overline g &= \phi(g)\phi(h)\overline g \psi(h)\psi(g)\\
		&= \phi(g)\phi(h)\phi(g)^{-1}\psi(g^{-1})\psi(h)\psi(g) \\
		&= \phi(g)\phi(h)\phi(g)^{-1}\psi(h)\\
		&= \phi(g)\phi(h)\phi(g)^{-1}\phi(h)^{-1}h\\
		&=[\phi(g),\phi(h)]h,
	\end{align*}
	which is in $H$ if and only if $[\phi(g),\phi(h)]\in H$. 
	
	Also,
	\begin{align*}
		\gamma_{\circ}(g)[h] &= \overline{g}\circ(gh)\\
		&=\phi(\overline g)gh\psi(\overline g)\\
		&=\phi(g)^{-1}gh\psi( \psi(g)g^{-1}\psi(g^{-1})) \tag{$\phi:(G,\circ)\to(G,\cdot)$ homomorphism }\\
		&=\phi(g)^{-1}gh\psi(g^{-1}) \tag{$\psi\in\Ab(G)$ }\\
		&=\phi(g)^{-1}gh\psi(h^{-1}g^{-1}h) \tag{$\psi\in\Ab(G)$}\\
		&= \psi(g)\phi(h)\psi(g^{-1})\psi(h)h^{-1}h\\
		&=[\psi(g),\phi(h)]h,
	\end{align*}
	which is in $H$ if and only if $[\psi(g),\phi(h)]\in H$.
	
	Of course, if $[G,\phi(H)]\le H$ then $[\phi(g),\phi(h)],\ [\psi(g),\phi(h)]\in H$ and $H$ is a strong left ideal of $(G,\circ,\cdot)$ and so (\ref{trois}) implies (\ref{un}). Conversely, if $H$ is a strong left ideal of $(G,\circ,\cdot)$ then   $[\phi(G),\phi(H)]$ and $[\psi(G),\phi(H)]$ are both contained in $H$, hence
	\[
		[g,\phi(h)] = [\psi(g)\phi(g^{-1})^{-1},\phi(h)]
		= \psi(g)[\phi(g^{-1})^{-1},\phi(h)]\psi(g^{-1})[\psi(g),\phi(h)].
	\]
	If we let $h_1= [\phi(g^{-1})^{-1},\phi(h)]$ and $h_2=[\psi(g),\phi(h)]$ then $h_1,h_2\in H$ and
	\begin{align*}
		[g,\phi(h)] &= \psi(g)h_1\psi(g^{-1})h_2	\\
		&= \psi(g)h_1\psi(h_1^{-1}g^{-1}h_1)h_1^{-1}h_1h_2\\
		&= [\psi(g),\phi(h_1)]h_1h_2\in H
	\end{align*}
	and hence $[G,\phi(H)]\le H$ as desired. This establishes the equivalence of (\ref{un}) and (\ref{trois}) and we are done. 
	\end{proof}

One the strong left ideals have been identified, \cite[Prop. 2.4]{Martin-LyonsTruman23} then gives the following.

\begin{corollary}\label{bigCirccor}
	Let $\psi\in\Ab(G)$, $H\le G$, and suppose that $[G,\phi(H)]\le H$. Define an action $\odot$ of $G$ on $G/H$ by $g\odot xH = gxH$. Then $(G,\cdot,G/H,\circ,\odot)$ and $(G,\cdot,G/H,\circ',\odot)$ are bracoids.
\end{corollary}

\begin{remark}\label{cosetshmoset}
As $H$ is a strong left ideal of $(G,\circ,\cdot)$ above our factor group should be $(G,\circ)/H$, hence a coset should be of the form $y\circ H$ for $y\in G$. However, 
\begin{align*}
	y\circ H &= \{y\circ h: h\in H\}\\
			&= \{y\psi(y^{-1})h\psi(y):h\in H\}\\
			&= \{y\psi(y^{-1})h\psi(h^{-1}yh)h^{-1}h:h\in H\}\\	
			&= \{y\psi(y^{-1})\phi(h)\psi(y)\phi(h)^{-1}h: h\in H\}		\\
			&= \{y[\psi(y^{-1}),\phi(h)]h:h\in H\}\\
			&=yH
\end{align*}	
since $[\psi(y^{-1}),\phi(h)]h\in H$.
\end{remark}

\begin{remark}\label{whocares}
	We have seen that for any $H$ such that $[G,\phi(H)]\le H$ we have $H\nrm (G,\circ)$. If in addition $H\nrm (G,\cdot)$ then $H$ is an ideal of the brace $(G,\circ,\cdot)$, and as a consequence $(G/H,\circ,\cdot)$ is also a brace. In the first bracoid $(G,\cdot,G/H,\circ,\odot)$ constructed above, notice that for $h\in H\nrm (G,\cdot)$ we have \[h\odot xH = (hxH)=x(x^{-1}hx)H = xH\] and hence $H$ acts trivially on $G/H$ and $\odot$ is not a faithful action, i.e., the bracoid is not \textit{reduced}: see \cite[Def. 2.14]{Martin-LyonsTruman23}. However, since $H$ acts trivially we get an induced action $G/H$ on $G/H$ by $gH\odot xH = gxH$ giving a reduced bracoid $(G/H,\cdot,G/H,\circ,\odot)$ which is the brace above. 
\end{remark}

Thus, the construction of strong left ideals (and their corresponding bracoids) reduces to finding subgroups of $(G,\circ)$ satisfying the commutator condition above. While many subgroups will not have this property, we present some general examples which do.

\begin{example}\label{norm}
	Let $G$ be any group, $\psi\in \Ab(G)$. Let $H=\ker \psi$. Then $\phi(H)=\{h\psi(h^{-1}):h\in H\} = H$ and $[G,\phi(H)]=[G,H]\le H$ since $\ker\psi\nrm G$. Thus $(G,\cdot,G/H,\circ,\odot)$ is a bracoid. However, since $H\nrm G$ we see that this bracoid reduces to the brace $(G/H,\circ,\cdot)$.
\end{example}

\begin{example}\label{fix}
	Let $G$ be any group, $\psi\in \Ab(G)$. Let $ H=\fix=\{h\in G: \psi(h)=h\}$. Clearly, $H=\ker\phi$, and hence $[G,\phi(H)]=\{e\}\le H$ and so $H$ is a strong left ideal of $(G,\circ,\cdot)$ and $(G,\cdot,G/H,\circ,\odot)$ is a bracoid. As $\phi:(G,\circ)\to(G,\cdot)$ is a homomorphism, we may identify $G/H=(G,\circ)/\ker\phi$ with $\phi(G)\le(G,\cdot)$ via the induced isomorphism $\widetilde \phi(xH)=\phi(x)$. In doing so we obtain an action $\odot'$ of $G$ on $\phi(G)$, namely
	\[g \odot' \phi(x) = \widetilde\phi(g\odot \widetilde\phi^{-1}\phi(x))=\widetilde\phi(g\odot xH) =\widetilde\phi(gxH)=\phi(gx).\]
		The result is the bracoid $(G,\cdot,\phi(G),\cdot,\odot)$ with $g\odot\phi(x)=\phi(gx)$.
	
	In contrast to Example \ref{norm}, it is not necessarily the case that our subgroup $\fix$ is normal in $(G,\cdot)$. In fact, for $g\in G,\ h\in \fix$ we have $ghg^{-1}\in\fix$ if and only if $\psi(ghg^{-1})=ghg^{-1}$. But $\psi(ghg^{-1})=\psi(h)=h$, so $\fix\nrm G$ if and only if $\fix\le Z(G)$ where $Z(G)$ is the center of $G$.
\end{example}

\begin{example}
	Generalizing Example \ref{fix}, let 
	\[\widehat H = \{h\in G: \phi(h)\in Z(G)\}.\]
	Then $\widehat H\le G$ and $\fix\le \widehat H$. Since $[G,\phi(\widehat H)]\le[G,Z(G)] = \{e\}\le \widehat H$ we get the bracoid $(G,\cdot,G/\widehat H,\circ,\odot)$. 
\end{example}
\begin{example}
		To provide a concrete example of each of the above, let $G=D_4=\gen{r,s:r^4=s^2=rsrs=e}$ and define $\psi:G\to G$ by $\psi(r)=rs,\psi(s)=e$. In this case, $\ker \psi=\gen{r^2,s},\ \fix = \gen{rs}$, and $\widehat H=\gen{r^2,rs}$.
\end{example}	

\begin{example}
	Let $H_1\le \fix$, and let $H=\ker\psi H_1$. Then $[G,\phi(H)] = [G,\ker\psi]$

	Some concrete examples of this can be found by adapting \cite[Ex. 2.4]{ColazzoKochMartin-LyonsTruman24}.
	\[G= \gen{x,y,z:x^{pq}=y^2=z^2=e,\ yxy=zxz=x^{-1},yz=zy}\cong C_{pq}\rtimes (C_2\times C_2)\]
	where $2<p<q$ are prime. The map $\psi:G\to G$ by $\psi(x^iy^jz^k)=y^jz^k$ is an endomorphism, and since $\psi(G)=C_2\times C_2$ we see that $\psi\in\Ab(G)$. Note $\ker \psi = \gen x$ and $\fix = \gen{y,z}$. By taking $H_1=\gen y,\ \gen z$, and $\gen{yz}$ we get three strong left ideals that are not found using the previous examples. 
\end{example}

\section{Many bracoids from $H=\fix$ and solutions to the {Y}ang-{B}axter equation}

In this section we will develop Example \ref{fix} a bit more. We will show how we can use \textit{brace blocks} to construct a (potentially large) family of bracoids from a single $\psi\in \Ab(G)$. Also we will show that, if we further insist that $\psi$ is idempotent, we obtain right non-degenerate solutions to the Yang-Baxter equation.

Let $\psi\in\Ab(G)$. We define a sequence of maps $\psi_n\in\Ab(G)$ recursively as follows: $\psi_0$ is trivial, and
\[\psi_n(g)=\psi(g)\psi_{n-1}(\phi(g)),\ n\ge 1, g\in G.\]
We immediately see that $\psi_1=\psi$. These maps are crucial to the work found in \cite{Koch22}, where it is shown that each $\psi_n$ is in fact an abelian map. Furthermore, if we define a family of binary operations $\{\circ_n:n\ge 0\}$ by \[g\circ_n h = g\psi_{n}(g^{-1})h\psi_n(g),\ g,h\in G\]
then for all $m,n\ge 0$ we have $(G,\circ_m,\circ_n)$ is a bi-skew brace. Each bi-skew brace constructed above is the case $m=0,n = 1$ (or vice versa).

We have seen that $H=\fix$ is a strong left ideal of $(G,\circ,\cdot)$. However, we can also show
\begin{proposition}\label{proppy}
	With notation as above, $H=\fix$ is a strong left ideal of $(G,\circ_n,\cdot)$ for all $n\ge 0$.
\end{proposition}

To prove this, we first require a lemma.

\begin{lemma}\label{lemmy}
	Let $\psi\in\Ab(G)$, and let $\{\psi_n:n\ge 0\}$ be the abelian maps constructed as above. For each $n\ge 0$ let $\phi_n(g)=g\psi_n(g^{-1})$. Then $\phi_n=\phi^n$ (where $\phi_1=\phi$ as usual).
\end{lemma}

\begin{proof} 
	[Proof (of \ref{lemmy}).] The result clearly holds for $n=0,1$. Suppose $\phi_{k-1}=\phi^{k-1}$. Then for $g\in G$ we have
	\begin{align*}
		\phi_k(g) &= g\psi_k(g^{-1})\\
			&= g \psi(g^{-1})\psi_{k-1}(\phi(g^{-1}))\\
			&=\phi(g)\psi_{k-1}(g^{-1}\psi(g))\\
			&=\phi(g)\psi_{k-1}(\psi(g)g^{-1})\\
			&=\phi(g)\psi_{k-1}(\phi(g)^{-1})\\
			&=\phi^{k-1}(\phi(g))\\
			&=\phi^k(g)
	\end{align*}
and the identity is established.
\end{proof}

\begin{proof} 
	[Proof (of \ref{proppy}).]
	Since $\psi_n\in\Ab(G)$, by Proposition \ref{bigCirc} it suffices to show that $[G,\phi_n(H)]\le H$, or equivalently by Lemma \ref{lemmy} that $[G,\phi^n(H)]\le H$. Since $H=\fix$ we have $\phi(H)=\{e\}$ and hence $\phi^n(H)=\{e\}$, so $[G,\phi^n(H)]=\{e\}\le H$ and $\fix$ is a strong left ideal of $(G,\circ_n,\cdot)$.  
\end{proof}

Thus, we may use $H=\fix$ to construct multiple bracoids.

\begin{corollary}
	Let $\psi\in \Ab(G)$, and $\psi_n,\circ_n$ as above. Then $(G,\cdot, \phi^n(G),\cdot,\odot_n)$ is a bracoid for all $n\ge 0$, where $g\odot_n \phi^n(x) = \phi^n(gx)$.
\end{corollary}

\begin{example}
	Let $G=D_4\times D_4 = \gen{r,s:r^4=s^2=rsrs=e}\times\gen{t,u:t^4=u^2=tutu=e}$ and define $\psi\in\Ab(G)$ by $\psi(r)=\psi(t)=e,\ \psi(s)=u,\ \psi(u)=s$. Then $\fix = \gen{su}$. It can be quickly computed that
	\[
		\phi(G)=\gen{r,t,su},\ \phi^n(G)=\gen{r,t} \text{ for } n\ge 2,
	\]
	thereby giving two bracoids.
\end{example}

While braces give non-degenerate solutions to the Yang-Baxter equation, bracoids in general do not. However, under special circumstances one can construct right non-degenerate solutions. 

\begin{proposition}\label{YBE1}
	Let $\psi\in\Ab(G)$ be idempotent. Then 
	\[ R(x,y)=(\psi(x)\phi(y)\psi(x^{-1}),\psi(x)\phi(y)^{-1}\phi(x^{-1})^{-1}y),\ x,y\in G \]
	is a right non-degenerate solution to the Yang-Baxter equation. 
\end{proposition}

\begin{proof}
	By \cite[Prop. 4.2]{ColazzoKochMartin-LyonsTruman24} it suffices to show that $(G,\cdot, \phi(G), \cdot,\odot)$ \textit{contains a brace}, that is, that there is a subgroup $K\le (G,\cdot)$ such that $(K,\cdot,\phi(G), \cdot, \odot)$  is a bracoid where $K$ acts regularly on $ \phi(G) $. However, here we can simply let $K=\phi(G)\le (G,\cdot)$. Then $(\phi(G),\cdot,\phi(G),\cdot,\odot)$ is a bracoid since the restriction of $\odot$ to $\phi(G)\le G$ is transitive: for $\phi(x),\phi(y)\in \phi(G)$ we have $\phi(y)\phi(x)^{-1}\in\phi(G)$ and
	\begin{align*}
		\left(\phi(y)\phi(x)^{-1}\right)\odot \phi(x) &= \phi\left(\phi(y)\phi(x)^{-1}\phi(x)\right)\\
		&= \phi(\phi(y))\\
		&= y\psi(y^{-1})\psi\left( y\psi(y^{-1})\right)\\
		&= y\psi(y^{-1})\psi(y)\psi(y^{-1})\\
		&=\phi(y).
	\end{align*}
	Thus, $(\phi(G),\cdot,\phi(G),\cdot,\odot)$ is a brace where $\phi(G)$ acts on itself regularly.
	
	As $(G,\cdot,\phi(G),\cdot,\odot)$ contains a brace we may obtain the precise solution following the explicit computations in \cite[\S 4]{ColazzoKochMartin-LyonsTruman24}. Explicitly, if we write $R(x,y) = (\lambda_x(y),\rho_y(x))$ then
	\begin{align*}
		\lambda_x(y) &= (x\odot e)^{-1}(x\odot (y\odot e)) \\
		&= \phi(x)^{-1}(x\odot \phi(y))\\
		&= \phi(x)^{-1} \phi(xy) \\
		&= (\psi(x)x^{-1})(xy\psi(xy^{-1}))\\
		&= \psi(x)\phi(y)\psi(x^{-1})
	\end{align*} 
	and 
	\[
		\rho_y(x) = (\lambda_x(y))^{-1}xy= (\psi(x)\phi(y)\psi(x^{-1}))^{-1} xy = \psi(x)\phi(y)^{-1}\phi(x^{-1})^{-1}y.
	\]
\end{proof}

\begin{remark}
	The map $R$ above is left-non-degenerate only if $\fix=\{e\}$, that is, $\psi$ is fixed-point-free. To see this, notice that if $y\in\fix$ then $\lambda_x(y) = e$, hence $\lambda_x$ cannot be injective if $\psi$ contains fixed points. Alternatively, observe 
	\[\lambda_x(y) = \psi(x)\phi(y)\psi(x^{-1})=\psi(x)x^{-1} xy\psi(y^{-1})\psi(x^{-1}) = \phi(x)^{-1}\phi(xy)^{-1}\in\phi(G),\]
	so the image of the first component of $r$ must be in $\phi(G)$. 
	
	On the other hand, if $\psi$ is fixed-point-free then $\fix$ is trivial, $\phi(G)=G$, and the bracoid is simply the brace $(G,\circ,\cdot)$. 
\end{remark}

Unfortunately, we can not call on Proposition \ref{proppy} to generate further solutions: notice that if $\psi\in\Ab(G)$ is idempotent then
\[\phi^2(g) = \phi(g)\psi(\phi(g)) = \phi(g) \]
and hence the bracoids $(G,\cdot,\phi(G),\cdot,\odot_1)$ and $(G,\cdot,\phi^2(G),\cdot,\odot_2)$ are the same.

\section{Strong left ideals of $(G,\cdot,\circ)$}

Having described the strong left ideals of $(G,\circ,\cdot)$, we now consider the other interpretation of the bi-skew braces constructed from abelian maps and find the strong left ideals of $(G,\cdot,\circ)$. As we will see, the condition that a subgroup be a strong left ideal of $(G,\cdot,\circ)$ is very easy to understand, however many of these turn out to be left ideals as well.

\begin{proposition}\label{bigDot}
	Let $\psi\in\Ab(G)$, and suppose that $H\nrm(G,\cdot)$. Then $H$ is a strong left ideal of both $(G,\cdot,\circ)$ and $(G,\cdot',\circ)$. Furthermore, $H$ is an ideal of each brace (hence, both braces)  if and only if $[G,\phi(H)]\le H$.
\end{proposition}

\begin{proof}
	In the first case, to show $H$ is a strong left ideal we require that $\gamma_{\bullet}: G\to \Perm(G)$ given by $\gamma_{\bullet}(g)[h]=g^{-1}(g\circ h)$ satisfies $\gamma_{\bullet}(g)[H]\subseteq H$. But for all $g\in G,\ h\in H$ we have
	\[\gamma_{\bullet}(g)[h] = g^{-1}(g\circ h) = g^{-1}\left(g\psi (g^{-1})h\psi(g)\right) = \psi(g^{-1})h\psi(g) \in H\] 
	by normality. Thus $H$ is a strong left ideal of $(G,\cdot,\circ)$. The proof that $H$ is a strong left ideal of $(G,\cdot',\circ)$ is similar since $g^{-1}\cdot'(g\circ h)=\phi(g)h\phi(g)^{-1}\in H$.
	
	Now $H$ is an ideal of $(G,\cdot,\circ)$ if and only if $H\nrm(G,\circ)$. We have
	\begin{align*}
		g\circ h \circ \overline g &= \phi(g)\phi(h)\phi(g)^{-1}\psi(g^{-1})\psi(h)\psi(g)\\
		&= \phi(g)\phi(h)\phi(g)^{-1}\psi(h)\\
		&= [\phi(g),\phi(h)]h,
	\end{align*}
	hence $H$ is an ideal of $(G,\cdot,\circ)$ if and only if $[\phi(G),\phi(H)]\le H$. But for any $g\in G,\ h\in H$ we have
	\[
		[\phi(g),\phi(h)] = [g\psi(g^{-1}),\phi(h)] = g[\psi(g^{-1}),\phi(h)][\phi(h),g^{-1}]g^{-1},\]
	and since $H\nrm(G,\cdot)$ we see that  $[\phi(g),\phi(h)]\in H$ if and only if $[\psi(g^{-1}),\phi(h)][\phi(h),g^{-1}]\in H$. But since
	\[[\psi(g^{-1}),\phi(h)] =\psi(g^{-1})h\psi(h^{-1})\psi(g)\psi(h)h^{-1}=\psi(g^{-1})h\psi(g)h^{-1}\in H\]
	by normality with respect to $ \cdot $ we see that $[\phi(g),\phi(h)]\le H$ if and only if $[\phi(h),g^{-1}]\in H$ and the conclusion quickly follows for $(G,\cdot,\circ)$. 
	
	Finally, let $[a,b]'$ denote the commutator in the group $(G,\cdot')$. Then $[a,b]'=[b^{-1},a^{-1}]$, hence 
	\[g\circ h \circ \overline g = h\cdot' [\phi(h)^{-1},\phi(g)^{-1}]'\] and the argument is similar to the one above.

\end{proof}

Once again, by \cite[Prop. 2.4]{Martin-LyonsTruman23} we get:

\begin{corollary}\label{bigDotcor}
	Let $\psi\in\Ab(G)$, $H\nrm G$. Define an action $\odot$ of $(G,\circ)$ on $G/H$ by $g\odot xH = (g\circ x)H$. Then $(G,\circ,G/H,\cdot,\odot)$ is a bracoid, as is $(G,\circ,G/H,\cdot',\odot)$. Furthermore, if $[G,\phi(H)]\not\le H$ then neither $(G,\circ,G/H,\cdot,\odot)$ nor $(G,\circ,G/H,\cdot',\odot)$  reduce to braces.
\end{corollary}

As with the previous case, we present some general examples.

\begin{example}\label{gencase}
	Let $G_1,G_2$ be groups and let $\alpha\in\Ab(G_1,G_2),\ \beta\in\Ab(G_2,G_1)$. Let $G=G_1\times G_2$, and define $\psi:G\to G$ by
	\[\psi(g_1,g_2) = (\beta(g_2),\alpha(g_1)),\ (g_1,g_2)\in G.\]
	It is easy to verify that $\psi\in\Ab(G)$, and hence $(G,\cdot,\circ)$ is a bi-skew brace with
	\[(g_1,g_2)\circ (h_1,h_2) = \left(g_1\beta(g_2^{-1})h_1\beta(g_2),g_2\alpha(g_1^{-1})h_2\alpha(g_1)\right).\]
	Clearly, by a slight abuse of notation, $G_1\nrm G$, and hence $(G,\circ,G/G_1,\cdot,\odot)$ is a bracoid with 
	\[(g_1,g_2)\odot (x_1,x_2)G_1 = \left(g_1\beta(g_2^{-1})x_1\beta(g_2),g_2\alpha(g_1^{-1})x_2\alpha(g_1)\right)G_1\ (g_1,g_2),(x_1,x_2)\in G.\]
	Of course, we can identify $G/G_1$ with $G_2$, thereby giving the bracoid $(G_1\times G_2,\circ,G_2,\cdot,\odot)$ with 
		\[(g_1,g_2)\odot x_2 = g_2\alpha(g_1^{-1})x_2\alpha(g_1),\ (g_1,g_2)\in G,\  x_2\in G_2.\]
	
	Also, we have
	\begin{align*}
		[(g_1,g_2),\phi(h_1,e)] &= [(g_1,g_2),(h_1,e)\psi(h_1,e)]\\
		&= [(g_1,g_2),(h_1,\alpha(h_1))]\\
		&= \left(g_1h_1g_1^{-1}h_1^{-1},g_2\alpha(h_1)g_2^{-1}\alpha(h_1^{-1})\right).
	\end{align*}  

Thus, $G_1$ is an ideal of $(G,\cdot,\circ)$ if and only if $\alpha(G_1)\le Z(G_2)$.

A similar bracoid is obtained starting with $G_2\nrm G$. 

\end{example}

\begin{remark}
	Observe that while $\beta$ plays a role in the construction of the brace $(G,\cdot,\circ)$ in Example \ref{gencase} it does not have any affect in the bracoid. Thus, if we are only interested in the bracoid constructed we can always take $\beta:G_2\to G_1$ to be trivial. In this case $\psi(g_1,g_2)=(e,\alpha(g_1))$ and 
	\[\fix = \{(g_1,g_2)\in G: (g_1,g_2)=(e,\alpha(g_1))\} = \{(e,e)\}.\]
	Thus, $\psi$ is a fixed-point free abelian map, and $\phi:(G,\circ)\to(G,\cdot)$ is an isomorphism. Thus, the bracoids produced in Example \ref{gencase} can always be obtained from a brace whose underlying groups are isomorphic.
	
	As we will see below, the choice of $\beta$ does play a role in the YBE solution we obtain.
\end{remark}

Generally, the isomorphism class of $(G,\circ)$ remains somewhat mysterious. However, we have
\begin{proposition}
	Let $\alpha,\beta\in \Ab(G),\ \psi$ as above. Then $\fix\subseteq Z(G,\circ)$.
\end{proposition}
\begin{proof}
	Let $(g_1,g_2)\in\fix$. Then we have, for $(h_1,h_2)\in G$,
	\begin{align*}
		(g_1,g_2)\circ (h_1,h_2) &= (g_1\beta(g_2^{-1}))h_1\beta(g_2),g_2\alpha(g_1^{-1}))h_2\alpha(g_1)) \\
		&= (g_1 g_1^{-1} h_1 g_1,g_2 g_2^{-1} h_2 g_2)\\
		&=(h_1g_1,h_2,g_2)
	\end{align*}
	while
	\begin{align*}
		(h_1,h_2)\circ (g_1,g_2) &= (h_1\beta(h_2^{-1}))g_1\beta(h_2),h_2\alpha(h_1^{-1}))g_2\alpha(h_1))\\
		 &= (h_1\beta(h_2^{-1})\beta(g_2)\beta(h_2),h_2\alpha(h_1^{-1})\alpha(g_1)\alpha(h_1))\\
		&=(h_1\beta(g_2),h_2\alpha(g_1))\\
		&=(h_1g_1,h_2,g_2).
	\end{align*}
\end{proof}

\begin{example}\label{permy}
	Let $G$ be any abelian group, and let $\Perm(G)$ be the group of permutations of $G$. Let $\alpha:G \to \Perm(G)$ be given by $\alpha(a) = \lambda(a)$ (that is, left regular representation), and let $\beta$ be trivial. Then $\alpha(A)$ is non-central, giving the bracoid $(G\times \Perm(G),\circ,\Perm(G),\cdot,\odot)$ with 
	\[(g,\sigma) \odot \tau = \sigma\lambda(g^{-1})\tau\lambda(g).\]
\end{example}

\begin{example}\label{rep}
	Let $G$ be any abelian group, and let $\varrho:G\to \GL_n(F)$ be a representation of $G$ for some field $F$. Then $(G\times \GL_n(F),\circ,\GL_n(F),\cdot,\odot)$ is a bracoid with 
	\[(g,A) \odot B = A\varrho(g^{-1})B\varrho(g).\]
\end{example}

\begin{example}\label{possiblyabridgetoofar}
	Let $\{G_i : i\in \Z_n\}$ and let $\alpha_i\in \Ab(G_i,G_{i+1})$ for each $i\in\Z_n$. Let $G=\prod_{i\in\Z_n} G_i$ and define $\psi\in\Ab(G)$ by 
	\[\psi\left( \prod_{i\in \Z_n} g_i\right)=\prod_{i\in \Z_n} \alpha_{i-1}(g_i). \]
	This gives a bi-skew brace $(G,\cdot,\circ)$, and if we let $H=G_0$ then 
	\begin{align*}
		\left[\left(\zprod g_i\right), \phi(h_0,0,\dots,0)\right] &= \left[\left(\zprod g_i\right), (h_0,0,\dots,0)\psi((h_0,0,\dots,0)^{-1})\right]\\
	&	= \left[\left(\zprod g_i\right), (h_0,\alpha_0(h_0),0,\dots,0)\right], 
	\end{align*}
	and we see that $[G,\phi(H)]\le H$ if and only if $\alpha_0(H)\in Z(G_1)$. Thus we get a bracoid which does not reduce to a brace if and only if $\alpha_0(G_0)\not\in Z(G_1)$.
	
\end{example}

\section{Solutions to the {Y}ang-{B}axter equation from Example \ref{gencase}}

We return to the case where $G=G_1\times G_2,\ \alpha\in\Ab(G_1,G_2),\ \beta\in\Ab(G_2,G_1)$ and $H=G_1$. Then $(G_1\times G_2,\circ,G_2,\cdot,\odot)$ is a bracoid with $(g_1,g_2)\odot x_2 = g_2\alpha(g_1)^{-1}x_2\alpha(g_1)$ as before. Since $(e,x_2)\circ (e,y_2) = (e,x_2y_2)$ we have $G_2\le (G_1\times G_2,\circ)$, and the action restricted to $G_2$ is simply
\[(e,g_2)\odot x_2 = g_2x_2.\]
This is evidently a transitive action, hence $(G_1\times G_2,\circ,G_2,\cdot,\odot)$ contains a brace. Applying \cite[Prop. 4.2]{ColazzoKochMartin-LyonsTruman24} will give us the following. 

\begin{proposition}\label{YBE2}
	With the notation above, write $\vec x=(x_1,x_2)$ and $\vec y=(y_1,y_2)$. Let  
	\begin{align*}\lambda_{\vec x}(\vec y) &= (e,\alpha(x_1^{-1})y_2\alpha(x_1))\\
		\rho_{\vec y}(\vec x)&=\left(\beta(y_2)x_1\beta(x_2^{-1})y_1\beta(x_2y_2^{-1}),\alpha(x_1)^{-1}y_2^{-1}\alpha(x_1)x_2\alpha(x_1)^{-1}y_2\alpha(x_1)\right)
		\end{align*}
	Then $R(\vec x, \vec y)=(\lambda_{\vec x}(\vec y), \rho_{\vec y}(\vec x))$ is a right non-degenerate solution to the Yang-Baxter equation.
\end{proposition}

\begin{proof}
	We simply use the technique of \cite[\S 4]{ColazzoKochMartin-LyonsTruman24}, adapting the notation since here since $G$ is viewed as a group under $\circ$. We have
	\begin{align*}
		\lambda_{\vec x}(\vec y) &= \left( e,(\vec x \odot e)^{-1}\right)\circ \left(e, (\vec x \odot \vec y\odot e)\right) \tag{since $\overline {(e,g_2)}=(e,g_2)^{-1}$}\\
		&=\left( (e, x_2^{-1})\right)  \circ \left(e, (\vec x \odot y_2)\right)\\
		&=  \left((e,x_2^{-1})\right) \circ \left(e, x_2\alpha(x_1)^{-1}y_2\alpha(x_1)\right) \\
		&= \left(e, \alpha(x_1)^{-1}y_2\alpha(x_1)\right)
	\end{align*}
	and
	\begin{align*}
		\rho_{\vec y}(\vec x) &= \overline{\lambda_{\vec x}(\vec y)} \circ \vec x \circ \vec y \\
		&=\left(e, \alpha(x_1)^{-1}y_2^{-1}\alpha(x_1)\right) \circ \left(x_1\beta(x_2^{-1})y_1\beta(x_2), x_2\alpha(x_1)^{-1}y_2\alpha(x_1)\right)\\
		&=\left(\beta(y_2)x_1\beta(x_2^{-1})y_1\beta(x_2y_2^{-1}),\alpha(x_1)^{-1}y_2^{-1}\alpha(x_1)x_2\alpha(x_1)^{-1}y_2\alpha(x_1)\right),
	\end{align*}
\end{proof}
giving the desired solution.

\begin{example}
	Return to Example \ref{rep}. Then $\alpha=\varrho$ and $\beta$ is trivial, giving
	\[R((g,A),(h,B)) = \left((e,\varrho(g^{-1})B\varrho(g)),(gh,\varrho(g)^{-1}B^{-1}\varrho(g)A\varrho(g^{-1})B\varrho(g) \right).\]
\end{example}

\begin{example}
	Let $G=C_8 \times S_4$ where $C_8 = \gen g$ is cyclic of order $8$. Let $\pi=(1234)\in S_4$, and define $\alpha\in\Ab(C_8,S_4)$ by $\alpha(i) = \pi^i$; furthermore, define $\beta\in\Ab(S_4,C_8)$ by 
	\[\beta(\sigma) = \begin{cases}
		e & \sigma\in A_4\\
		g^4 & \sigma\notin A_4
	\end{cases}. \]
	For brevity we will write $g^i\sigma$ for the element $(g^i,\sigma)$. The resulting YBE solution is
	\[R(g^i\sigma,g^j\tau) = \begin{cases}
		\left(\pi^{-i}\tau\pi^i, g^{i+j}\pi^{-i}\tau^{-1}\pi^i\sigma\pi^{-i}\tau\pi^i \right) & \tau\in A_4\\
		\left(\pi^{-i}\tau\pi^i, g^{i+j+4}\pi^{-i}\tau^{-1}\pi^i\sigma\pi^{-i}\tau\pi^i \right) & \tau\notin A_4
	\end{cases}.\]
\end{example}

	The works of \cite{Koch21,Koch22} construct braces starting from a nonabelian group $G$. Strictly speaking, that $G$ be nonabelian is not necessary, however if $(G,\cdot)$ is abelian and $\psi\in\End(G)$ we have
	\[g\circ h = g\psi(g^{-1})h\psi(g)=gh\]
	and hence the brace obtained $(G,\cdot,\cdot)$ is trivial.
	
	We conclude this paper by observing that (abelian) maps on abelian groups can give interesting solutions to the Yang-Baxter equation.

	Let $G$ be an abelian group, and let $\psi\in\End(G)=\Ab(G)$. Then $\phi$ is also an endomorphism, and $\phi(G)\le G$ since $\phi(x)\phi(y)=\phi(xy)$. This gives the bracoid $(G,\cdot,\phi(G),\cdot,\odot)$ where $g\odot\phi(x) = \phi(g)\phi(x)$. In this case, of course, $\fix$ is an ideal of $(G,\cdot,\phi(G),\cdot,\odot)$ 
	
	Now suppose $\psi$ is idempotent. By Proposition \ref{YBE1} we get 
	\[R(x,y)= (\phi(y),\psi(y)x) \]
	is a solution to the Yang-Baxter equation. Additionally, as $\phi\in\Ab(G)$ and $\psi(x) = x\phi(x^{-1})$ we obtain an additional solution
	\[R'(x,y)= (\psi(y),\phi(y)x). \]
	
	Thus, idempotent maps on abelian groups can be used to find solutions.
	
%	\ak{Am I the only person who thinks the following is interesting? If we compose the solutions (writing $RR'$ and $R'R$ for them) we get
%	\[
%			R'R=(\psi\times \phi)\Delta\mu, \
%			RR'= (\phi\times \psi)\Delta\mu, \
%			R'R=\tau RR',
%	\]
%	where $\mu$ is multiplication, $\Delta$ is ``comultiplication'', i.e., $\Delta(a)=(a,a)$, and $\tau\in\End(G\times G)$ is the twist map? 
%	
%	And, unless I screwed up the computations, $RR'$ and $R'R$ are extremely degenerate YBE solutions. 
%		
%		I am probably the only one who finds it interesting.}
%	
%	

\bibliographystyle{alpha} 
\bibliography{MyRefs}
\end{document}